\documentclass[final, leqno]{amsart}

\usepackage{amsmath,amsthm}%,amsfonts}
\usepackage {latexsym}
\usepackage{amssymb}
\usepackage{color}
\newcommand{\eps}{{\varepsilon}}
\newcommand{\R}{{\mathbb R}}

\newcommand{\sph}{{\mathbb S}}

\newcommand{\Compl}{{\mathbb C}}

\newcommand{\les}{\lesssim}

\newcommand{\sigmacheck}{\check{\sigma}}
\newcommand{\la}{\langle}
\newcommand{\ra}{\rangle}

\def\norm[#1][#2]{\|#1\|_{#2}}
\def\bignorm[#1][#2]{\big\|#1\big\|_{#2}}
\def\Bignorm[#1][#2]{\Big\|#1\Big\|_{#2}}
\def\japanese[#1]{\langle #1 \rangle}
\def\Im[#1]{{\rm Im}(#1)}
\def\Re[#1]{{\rm Re}(#1)}

\newtheorem{theorem}{Theorem}

\newtheorem{prop}[theorem]{Proposition}
\newtheorem{proposition}[theorem]{Proposition}

\theoremstyle{remark}
\newtheorem{remark}{Remark}

\begin{document}

\title[The Helmholtz equation with $L^p$ data]
{The Helmholtz equation with $L^p$ data and Bochner-Riesz multipliers}

\date{February 4, 2015}

\author{Michael\ Goldberg}
\thanks{This work is supported in part by Simons Foundation Grant \#281057.}
\address{Department of Mathematics, University of Cincinnati,
Cincinnati, OH 45221-0025}
\email{Michael.Goldberg@uc.edu}

\begin{abstract}
We prove the existence of $L^2$ solutions to the Helmholtz equation
$(-\Delta - 1)u = f$ in $\R^n$ assuming the given data $f$
belongs to $L^{(2n+2)/(n+5)}(\R^n)$ and satisfies the
``Fredholm condition" that $\hat{f}$ vanishes on the unit sphere.
This problem, and similar results for the perturbed Helmholtz
equation $(-\Delta -1)u = -Vu + f$, are connected to the Limiting Absorption
Principle for Schr\"odinger operators.

The same techniques are then used to prove that a wide range of
$L^p \mapsto L^q$ bounds for Bochner-Riesz multipliers
are improved if one considers their action on the closed subspace
of functions whose Fourier transform vanishes on the unit sphere.
\end{abstract}

\maketitle

We consider the existence of a well-defined solution map for the
Helmholtz equation in Euclidean space
\begin{equation} \label{eq:Helmholtz}
\left\{ \begin{aligned}
(-\Delta -1)u &= f  \text{ in } \R^n \\
u&\in L^2(\R^n)
\end{aligned} \right.
\end{equation}
By conjugating with dilations, the same problem can be posed with an operator
$(-\Delta - \lambda^2)$, $\lambda > 0$ with minimal modification.
These equations are translation invariant, so it would be desirable to choose
$f$ from a function space whose norm is also translation invariant.
Our goal is to establish existence of solutions and a norm bound for $u$ in
terms of the $L^p$-norm of the given data $f$, provided $f$ is formally
orthogonal to all plane waves of unit frequency.

The Fourier dual formulation of~\eqref{eq:Helmholtz} is 
\begin{equation} \label{eq:multiplier}
\left\{ \begin{aligned}
\hat{u}(\xi) &= \frac{\hat{f}(\xi)}{|\xi|^2 - 1} \\
\hat{u} &\in L^2(\R^n)
\end{aligned} \right.
\end{equation}
with respect to the definition
$\hat{f}(\xi) = \int_{\R^n} e^{-i \xi \cdot x}f(x)\,dx$.
The corresponding Plancherel identity is
$\norm[\hat{u}][2] = (2\pi)^{n/2}\norm[u][2]$.

It is immediately clear from~\eqref{eq:multiplier} that solutions should be
unique, as $|\xi|^2 - 1$ is nonzero almost everywhere and the Fourier Transform
is (a scalar multiple of) a unitary map between $L^2(dx)$ and $L^2(d\xi)$.
One can also infer that solutions exist only if $\hat{f}$ vanishes on the unit
sphere in a suitable sense, and also the restrictions of $\hat{f}$ to the sphere
of radius $r$ must be controlled as $r$ approaches $1$.

It would be sufficient, for example, if the map
$S(r) = \hat{f}(r\,\cdot\,)\big|_{\sph^{n-1}}$
(taking $\R_+$ into $L^2(\sph^{n-1})$)
was H\"older continuous of order $\gamma > \frac12$ at $r=1$
and vanished there.  Then the scalar restriction function
\begin{equation} \label{eq:F}
F(r) = \norm[\hat{f}][L^2(r\sph^{n-1})]^2
\end{equation}
would be $O(|r-1|^{2\gamma})$, and the formula
\begin{equation} \label{eq:Unorm}
\norm[\hat{u}][2]^2 = \int_0^\infty \frac{F(r)}{(r^2-1)^2}\,dr
\end{equation}
would be locally integrable at $r=1$.

In fact the desired continuity can be achieved if $\hat{f}$ belongs to the
Sobolev space $W^{1,2}(\R^n)$.  While $S(r)$ is only H\"older continuous of order
exactly $\frac12$, the one-dimensional Hardy inequality suffices to establish
integrability of~\eqref{eq:Unorm}.  This argument plays a central
role in Agmon's bootstrapping method for the decay of eigenfunctions
of a Schr\"odinger operator~\cite{Ag75}.  For the Helmholtz equation in
particular the following result is proved there.
\begin{theorem}[Agmon] \label{thm:Agmon}
Suppose $(1+|x|)^\beta f \in L^2(\R^n)$ for some $\beta > \frac12$, and
$\hat{f}$ vanishes on the unit sphere in the $L^2$-trace sense.  Then
there exists a unique function $u$ such that
$(-\Delta -1)u = f$ and $(1+|x|)^{\beta - 1}u \in L^2(\R^n)$.
\end{theorem}

It is not obvious that a similar result should hold for $f \in L^p(\R^n)$
without weights, 
regardless of the exponent, as an $L^p$ condition typically doesn't guarantee
that $S(r)$ is H\"older continuous of any positive order.  Nevertheless
an $L^2$ solution operator for the Helmholtz equation exists for data
in a narrow range of $L^p$ spaces.

\begin{theorem} \label{thm:Helmholtz}
Let $n \ge 3$ and $\max(1, \frac{2n}{n+4}) \leq p \leq \frac{2n+2}{n+5}$,
with $(n, p) \not= (4,1)$.  Suppose $f \in L^p(\R^n)$ and
$\hat{f}$ vanishes on the unit sphere in the $L^2$-trace sense.  

There exists a unique $u \in L^2(\R^n)$
such that $(-\Delta -1)u = f$.  Furthermore, $\norm[u][2] \leq C_{n,p}\norm[f][p]$.
\end{theorem}

There is no statement in dimensions $1$ or $2$ because
$\frac{2n+2}{n+5} < 1$.  When $n=1$ it should suffice to allow
$e^{\pm i |x|}f$ to belong to the Hardy space $H^{2/3}(\R)$.
It is less clear what cancellation conditions might be required for $n=2$.

The lower exponent bound of $\frac{2n}{n+4}$ comes from Sobolev embedding.
It can be disregarded if one applies any sort of cutoff to remove high
frequencies.  As a special case, the sharp cutoff at $|\xi| = 1$ leaves 
a Bochner-Riesz multiplier of order -1.  For further discussion of these
operators we adopt the definition
\begin{equation} \label{eq:B-R}
(S^\alpha f)\hat{\phantom{i}}(\xi) = (1-|\xi|^2)_+^\alpha \hat{f}(\xi).
\end{equation}
%Note that the gamma function is absent from the definition.
For $\alpha \le -1$ we define $S^{\alpha}$ by (formal) positivity
of the operator rather than by analytic continuation.  This preserves the
multiplicative structure $S^\alpha S^\beta = S^{\alpha + \beta}$,
however it comes at the cost that $S^{\alpha}$ will not have a bounded action
on general Schwartz functions once $\alpha \le -1$.

Never the less, $S^\alpha$ may behave well when applied to functions whose
Fourier transform vanishes on the unit sphere, as stated below.

\begin{theorem} \label{thm:B-R}
Let $n \ge 2$ and $\frac12 \leq \alpha < \frac32$.
Suppose $f \in L^p(\R^n)$, $1 \leq p \leq \frac{2n+2}{n+1+4\alpha}$
with $(\alpha, p) \not= (\frac12, \frac{2n+2}{n+3})$, and
suppose $\hat{f}$ vanishes on the unit sphere.
% in the $L^2$-trace sense.

Then $\norm[S^{-\alpha}f][2] \les \norm[f][p]$.
\end{theorem}

Both Theorems~\ref{thm:Helmholtz} and~\ref{thm:B-R} are easily derived from
the following statement, which is our main technical result.

\begin{prop} \label{prop:main}
Let $n \geq 2$ and $\frac12 < \alpha < \frac32$.  Suppose 
$f \in L^p(\R^n)$, $1 \leq p \leq \frac{2n+2}{n+1+4\alpha}$.
%$f \in L^{\frac{2n+2}{n+1+4\alpha}}(\R^n)$
%and $\hat{f}$ vanishes on the unit sphere.
There is a constant $C_\alpha$ such that
\begin{equation} \label{eq:main}
\bigg| \int_{\frac12 < |\xi| < \frac32} \frac{|\hat{f}(\xi)|^2}{((1-|\xi|^2)^2 + \eps^2)^{\alpha}}\,d\xi
- \frac{C_\alpha}{\eps^{2\alpha-1}} \norm[\hat{f}][L^2(\sph^{n-1})]^2 \bigg|
\les  \norm[f][p]^2
\end{equation}
with a constant that remains bounded in the limit $\eps \to 0$.
\end{prop}

In both theorems, it is given that $\hat{f}$ vanishes on the unit sphere,
eliminating the
$\eps^{1-2\alpha}\norm[\hat{f}][L^2(\sph^{n-1})]^2$ term from the left
side of~\eqref{eq:main}.
Assuming Proposition~\ref{prop:main} holds, the same inequality is then
true with $\eps = 0$ by monotone convergence.
The Hausdorff-Young inequality is more than sufficient to bound the left-side
integral over the center region $\{|\xi| <\frac12\}$ for any $f \in L^p$,
$1 \leq p \leq 2$.

For Theorem~\ref{thm:Helmholtz}, let $\chi \in C^\infty_c(\R^n)$ be any smooth 
cutoff that is identically 1 in the ball $\{|\xi| \leq \frac54\}$ and
has support in the ball of radius $\frac32$.
Theorem~\ref{thm:Helmholtz} then reduces to the $\alpha =1 $ case
of Proposition~\ref{prop:main} combined with a Sobolev embedding estimate
for the high frequency tail $\frac{1-\chi}{|\xi|^2 - 1}\hat{f}$.
In a similar manner, all cases
of Theorem~\ref{thm:B-R} with $\alpha > \frac12$ follow from the Proposition by
applying the multiplier of the unit ball, which is bounded on $L^2(\R^n)$.

Finally, if $\alpha = \frac12$ and $p \in [1, \frac{2n+2}{n+3})$,
we have already established Theorem~\ref{thm:B-R} for the pair $(\beta, p)$
with $\beta = \min\big((n+1)(\frac{1}{2p} - \frac14), 1\big)$.  Since 
$\beta > \frac12$, it follows that 
$\norm[S^{-1/2}f][2] \leq \norm[S^{-\beta}f][2]$ by Plancherel's formula.

Sharpness of the upper exponent $\frac{2n+2}{n+1+4\alpha}$ is verified
using a Knapp-type example.  Let $\hat{f}$ be a smooth compactly supported
function, suitably scaled to have support in the slab 
$\{|\xi'| \leq \delta, \;1-2\delta^2 \leq \xi_n \leq 1-\delta^2\}$, where
$\xi' = (\xi_1, \xi_2, \ldots , \xi_{n-1}) \in \R^{n-1}$ and unit height. 
Then $|f(x)| \sim \delta^{n+1}$ over the dual region 
$\{|x'| \leq \delta^{-1},\, |x_n| \leq \delta^{-2}\}$ and has rapid decay
elsewhere.  It follows that $\norm[f][p] \sim \delta^{(n+1)(1-p^{-1})}$
and $\norm[S^{-\alpha}f][2] \sim \delta^{\frac{n+1}{2}-2\alpha}$.  If
$p > \frac{2n+2}{n+1+4\alpha}$ then $(n+1)(1-p^{-1}) > \frac{n+1}{2} - 2\alpha$
and Theorem~\ref{thm:B-R} fails by taking $\delta$ to zero.

\begin{remark}
When $0 < \alpha < \frac12$, no vanishing condition on the unit sphere is
needed in the statement of Theorem~\ref{thm:B-R}.
The range of viable exponents is once again
$p \in [1, \frac{2n+2}{n+1+4\alpha}]$ including the endpoints~\cite{BaMcOb95}.
The full range of $L^p \to L^q$ mappings in this regime is established
in~\cite{Ba97}.
\end{remark}

\begin{remark}
The statement of Proposition~\ref{prop:main} is not true for $\alpha > 1$
if integration is
limited to the inner annulus $\{\frac12  < |\xi| < 1\}$.  An additional remainder
term of order $\eps^{2-2\alpha}$ is present in that case.  The same remainder
term appears with the opposite sign if one integrates over the outer annulus
$\{1 < |\xi| < \frac32\}$.  This illustrates a difference in behavior between
``one-sided" and "two-sided" Bochner-Riesz multipliers of order below $-1$,
with the former being modestly more singular than the latter.
\end{remark}

\begin{remark}
The endpoint case $\alpha = \frac12$, $p = \frac{2n+2}{n+3}$ is quite delicate.
The conclusion is certainly false if one does not assume that $\hat{f}$
vanishes on the unit sphere.
In one dimension it remains false even with the vanishing condition.
Since $S^{-\frac12}$ in one dimension
is closely related to the fractional integral operator $I_{1/2}$, 
a stronger condition that $e^{\pm i x}f$
belongs to the Hardy space $H^1(\R)$ is needed to guarantee that
$S^{-1/2}f \in L^2(\R)$.

The one-dimensional counterexamples do not generalize well to $n \geq 2$.
We believe it is an open problem whether Theorem~\ref{thm:B-R} is true 
in these endpoint cases.
\end{remark}

It is possible to extend Theorem~\ref{thm:B-R} further by interpolation 
with other known estimates for Bochner-Riesz operators, subject to a few
technical limitations.  In this paper we do not assemble a full catalog
of such estimates but instead consider a family of
bounds that are sharp with respect to Knapp counterexamples.

\begin{theorem} \label{thm:B-Rextension}
Let $n \geq 2$ and $\beta \in (\frac12, \frac32)$ with 
$\beta \leq \frac{n+1}{4}$ .  Suppose 
$f \in L^{\frac{2n+2}{n+1+4\beta}}(\R^n)$ and $\alpha \in
[\beta, 2\beta]$ with $\alpha < 2$.  Then
\begin{equation}
\norm[S^{-\alpha} f][\frac{2n+2}{n+1-4(\alpha-\beta)}]
\les \norm[f][\frac{2n+2}{n+1+4\beta}].
\end{equation}
\end{theorem}

\begin{remark}
The restriction $\alpha < 2$ may be removed if one instead
considers the analytic family of operators $\tilde{S}^{-\alpha}
 = \Gamma(1-\alpha)^{-1}S^{-\alpha}$.  This will be evident
in the proof.
\end{remark}

Theorem~\ref{thm:Agmon} plays an important role in the spectral theory of
Schr\"odinger operators $H = -\Delta + V(x)$ with a short-range 
potential.  Namely, it is used in a bootstrapping
argument to show that any singular part of the essential spectrum of $H$
must contain embedded eigenvalues.  Thus the spectral measure on
compact subsets of $[0,\infty) \setminus \sigma_{pp}(H)$ is absolutely continuous
and satisfies an assortment of uniform mapping properties.
In section~\ref{sec:embedded} we present a similar bootstrapping application
using Theorem~\ref{thm:Helmholtz} as the primary device.  These results are
contained within the more general Limiting Absorption Principle of Ionescu and 
Schlag~\cite{IoSc06}, and serve as an instructive special case.

The discussion of perturbed Schr\"odinger operators naturally raises the
question of whether there is a similar existence theorem for the
Helmholtz equation $(-\Delta + V - 1)u = f$.  Using resolvent
identities we are able to prove the following.

\begin{theorem} \label{thm:Helmholtz2}
Let $n \ge 3$, $p = \frac{2n+2}{n+5}$, and suppose 
$V \in L^{\frac{n+1}{2}}(\R^n)$. There is a subspace
$X \subset L^p(\R^n)$, isomorphic to the subspace 
$X_0 \subset L^p(\R^n)$ of functions whose Fourier transform vanishes
on the unit sphere, with the following property: For each $f \in X$
there exists a unique $u \in L^2(\R^n)$
such that $(-\Delta +V -1)u = f$.  Furthermore, $\norm[u][2] \leq C_{n}\norm[f][p]$.
\end{theorem}

\begin{remark}
The integrability condition $V \in L^\frac{n+1}{2}(\R^n)$ appears as a sharp
threshold for short-range potentials in both~\cite{IoSc06} and~\cite{KoTa06}.
For $1 \leq p < \frac{2n+2}{n+5}$ a stronger set of constraints on $V$
may be required.
\end{remark}

\section{Proof of Proposition~\ref{prop:main}}
The proof of Proposition~\ref{prop:main} mirrors that of the
sharp Stein-Tomas restriction theorem.  We follow the exposition
in~\cite{ScMu13} most closely.  

Let $\sigma_r$ denote the surface measure on $r\sph^{n-1}$
inherited from its embedding in $\R^n$.  
%By assumption
%$\hat{f}$ vanishes on the unit sphere, which makes
%$\sigmacheck_1 * f = 0$.
%The lefthand expression in~\eqref{eq:main} is equal to the inner product
%$\la (K_0^\eps - C_\alpha\eps^{1-2\alpha}\sigmacheck_1) *f, f\ra$.
The main estimate will be a bound on $\la K_1^\eps * f, f\ra$, where
\begin{equation} \label{eq:K1eps}
K_1^\eps = \int_\frac12^\frac32 
\frac{\sigmacheck_r - \sigmacheck_1}{((1-r^2)^2+\eps^2)^{\alpha}}\,dr
= \int_{-\frac12}^\frac12
\frac{\sigmacheck_{1+s} - \sigmacheck_1}{(s^2(2+s)^2 + \eps^2)^{\alpha}}
\,ds.
\end{equation}
This is almost equal to the lefthand expression in~\eqref{eq:main}, with the
only discrepancy arising in the coefficient of $\la \sigmacheck_1 * f, f\ra$.
More precisely,
\begin{align*}
[\text{Left side of }\eqref{eq:main}] - \la K_1^\eps * f,  f\ra
 &= \Big(\int_\frac12^\frac32 ((1-r^2)^2+\eps^2)^{-\alpha}\,dr - 
\frac{C_\alpha}{\eps^{2\alpha-1}}\Big)\la\sigmacheck_1*f,f\ra \\
&= O(1)\la\sigmacheck_1 * f,f\ra.
\end{align*}
Since $f \in L^p(\R^n)$
with $p \leq \frac{2n+2}{n+3}$, the $O(1)$ term can be absorbed into the right
side of~\eqref{eq:main} by the Stein-Tomas theorem.

The integrand in~\eqref{eq:K1eps}
may become highly singular at $s=0$ as $\eps$ decreases.
However the denominator is approximately an even function of $s$, 
while the leading order behavior of the numerator is an odd function.
To be precise, let
\begin{align*}
A_{even}(s) &= \frac12 \left(\frac{1}{(s^2(2+s)^2 + \eps^2)^\alpha} + 
\frac{1}{(s^2(2-s)^2+ \eps^2)^\alpha}\right) \\
A_{odd}(s) &= \frac12\left(\frac{1}{(s^2(2+s)^2 + \eps^2)^\alpha} - 
\frac{1}{(s^2(2-s)^2+ \eps^2)^\alpha}\right).
\end{align*}
Then
\begin{equation} \label{eq:Kevenodd}
K_1^\eps = \frac12 
\int_{-\frac12}^\frac12 \Big(A_{even}(s)(\sigmacheck_{1+s} - 2\sigmacheck_1 +
\sigmacheck_{1-s}) + A_{odd}(s)(\sigmacheck_{1+s} - \sigmacheck_{1-s})\Big)\,ds
\end{equation}
The main size bounds for $A_{even}$ and $A_{odd}$ are:
\begin{equation}\label{eq:Aevenodd}
|A_{even}(s)| \les s^{-2\alpha}, \quad |A_{odd}(s)| \les s^{1-2\alpha}
\quad \text{uniformly in } \eps > 0.
\end{equation}

It is a common practice to estimate the inverse Fourier transform of a
surface measure by decomposing the surface into smaller regions where
stationary phase methods can be applied.
Consider a conical decomposition $\sum_{j=1}^n \eta_j(\frac{\xi}{|\xi|}) = 1$
where each smooth cutoff $\eta_j$
is supported in the region where $|\xi_j| \sim |\xi|$.
One may symmetrize so that each $\eta_j$ is invariant under
reflections across any one of the coordinate planes.
 Then~\eqref{eq:K1eps}
splits into a directional sum
\begin{equation} \label{eq:K1decomp}
K_1^\eps = 
\sum_{j=1}^n \int_{-\frac12}^\frac12
\frac{(\eta_j\sigma_{1+s})\check{\!\phantom{i}} 
- (\eta_j\sigma_1)\check{\!\phantom{i}}}{(s^2(2+s)^2 + \eps^2)^{\alpha}} \,ds.
\end{equation}
Let $K_2^\eps$ denote the $j=n$ term of this sum and write coordinates in $\R^n$ as
$(x', x_n)$ or $(\xi', \xi_n)$.  We will make further
estimates on $K_2^\eps$ as a representative element.

Inside the support of $\eta_n\sigma_r$, the relationship
$\xi_n = \pm(r^2 - |\xi'|^2)^{1/2}$ expresses $\xi_n$ as a smooth function
of $\xi'$ on each hemisphere.  Then the inverse Fourier transform of
$\eta_n\sigma_r$ takes the form
\begin{equation*}
(\eta_n\sigma_r)\check{\!\phantom{i}}(x',x_n) = (2\pi)^{-n}\sum_\pm
\int_{\R^{n-1}}
\frac{r\,\eta_n\big(\frac{\xi'}{r}, \pm\frac{\sqrt{r^2-|\xi'|^2}}{r}\big)
}{\sqrt{r^2-|\xi'|^2}} e^{i(x'\cdot\xi' \pm x_n\sqrt{r^2-|\xi'|^2})}\,d\xi'.
\end{equation*}
For $\frac12 < r <\frac32$, the Hessian of the phase function is bounded below
by $x_n$ times the $(n-1)$-identity matrix and the initial fraction is a
uniformly smooth function.  This leads to the pointwise bound
\begin{equation*}
|(\eta_n\sigma_r)\check{\!\phantom{i}}(x',x_n)| \les (1+ |x_n|)^{\frac{1-n}{2}}.
\end{equation*}
for $r$ in this range.  Furthermore one can differentiate with respect to $r$
under the integral sign to obtain bounds
\begin{equation}
|\partial_r^k(\eta_n\sigma_r)\check{\!\phantom{i}}(x',x_n)| \les 
(1+|x_n|)^{\frac{1-n}{2}+k}.
\end{equation}

Taylor remainder estimates then imply that
\begin{align*}
|(\eta_n\sigma_{1+s})\check{\!\phantom{i}}(x) -
(\eta_n\sigma_{1-s})\check{\!\phantom{i}}(x)| 
&\les \min(|s|(1+|x_n|),1)(1+|x_n|)^{\frac{1-n}{2}}, \\
|(\eta_n\sigma_{1+s})\check{\!\phantom{i}}(x) 
- 2(\eta_n\sigma_1)\check{\!\phantom{i}}(x) 
+ (\eta_n\sigma_{1-s})\check{\!\phantom{i}}(x)| 
&\les \min(s^2(1+|x_n|)^2,1)(1+|x_n|)^{\frac{1-n}{2}}
\end{align*}
while $|s| < \frac12$.  Plugging these and~\eqref{eq:Aevenodd} into the
appropriately modified version of~\eqref{eq:Kevenodd}, one concludes that
\begin{equation}
|K_2^\eps(x)| \les (1+|x_n|)^{\frac{4\alpha-1-n}{2}}.
\end{equation}

In other words,  for a fixed 
choice of $x_n$, the restricted convolution operator
\begin{equation*}
Tg(x') = \int_{\R^{n-1}} K_2^\eps(x' - y', x_n)g(y') \, dy'
\end{equation*}
maps $L^1(\R^{n-1})$ to $L^\infty(\R^{n-1})$ with operator norm
controlled by $(1+|x_n|)^{\frac{4\alpha-1-n}{2}}$.

One can also determine the size of $T$ as an operator on $L^2(\R^{n-1})$.
This bound is given by the essential supremum of the $x'$-Fourier transform of the
convolution kernel $K_2^\eps$.  
Since $K_2^\eps$ is a superposition of the inverse Fourier transforms
of $(\eta_n \sigma_s)$ as in~\eqref{eq:K1decomp}, the $x'$-Fourier
transform reverses the procedure in all except the $x_n$ variable.
More precisely,
\begin{equation*}
\int_{\R^{n-1}} e^{-i\xi'\cdot x'}K_2^\eps(x',x_n)\,dx'
= \frac{1}{2\pi}\int_{\{\xi'\}\times\R}\int_{-\frac12}^{\frac12} 
\frac{e^{i x_n\xi_n}(\sigma_{1+s} - \sigma_1)\eta_n}{
(s^2(2+s)^2 + \eps^2)^{\alpha}} \,ds\,d\xi_n.
\end{equation*}

  If the integral over $s$ is split into even and odd contributions
as in~\eqref{eq:Kevenodd}, the result is
\begin{align*}
\int_{\R^{n-1}} e^{-i\xi'\cdot x'} &K_2^\eps(x',x_n)\,dx' \\
&= \begin{aligned}[t]
\frac{1}{4\pi}\int_{-\frac12}^{\frac12} \bigg( 
&A_{even}(s)  \int_{\{\xi'\}\times \R} 
e^{i x_n\xi_n}(\sigma_{1+s} - 2\sigma_1 + \sigma_{1-s})\eta_n\,d\xi' \\
+\ &A_{odd}(s) \int_{\{\xi'\}\times\R}e^{ix_n\xi_n}(\sigma_{1+s} - \sigma_{1-s})
\eta_n\,d\xi'\bigg) \,ds
\end{aligned}
\end{align*}

For a fixed choice of $\xi' \in \R^{n-1}$ and radius $r>0$, the line
$\{\xi'\} \times \R$ intersects the support of $\sigma_r$ 
only when $\xi_n = \pm\sqrt{r^2 - |\xi'|^2}$.  Thus for $|\xi'|<r$
\begin{equation*}
\int_{\{\xi'\}\times\R} e^{ix_n\xi_n}\sigma_r\eta_n\,d\xi'
= \frac{2\cos\big(x_n\sqrt{r^2-|\xi'|^2}\big)\,
\eta_n\Big({\textstyle \frac{\xi'}{r}},\sqrt{1-(|\xi'|/r)^2}\Big)}{
\sqrt{1-(|\xi'|/r)^2}}
\end{equation*}
and is zero otherwise.  The denominator accounts for the angle of
intersection between the line and surface.  
It is bounded away from zero within the support
of $\eta_n$, so the integral expression is a smooth bounded
function of $\xi'$ and $r$.
Within the range $\frac12 < r < \frac32$, its first two 
derivatives with respect to $r$ are bounded by $(1+|x_n|)$ and 
$(1+|x_n|)^2$
respectively.  It follows that
\begin{align*}
\Big| \int_{\R^{n-1}} e^{-i\xi'\cdot x'} K_2^\eps(x',x_n)\,dx' \Big|
&\les \int_{-\frac12}^{\frac12} 
\begin{aligned}[t] A_{even}(s) &\max(s^2(1+|x_n|)^2,1)\\
 &+ A_{odd}(s) \max(|s|(1+|x_n|),1)\,ds \end{aligned}\\
&\les (1+|x_n|)^{2\alpha-1}
\end{align*}
and therefore $T$ is a bounded operator on $L^2(\R^{n-1})$ with norm comparable
to $|x_n|^{2\alpha-1}$.  Interpolating with the previous $L^1\to L^\infty$ bound
shows that
\begin{equation*}
\norm[Tg][L^{p'}(\R^{n-1})] \les (1+|x_n|)^{2\alpha + \frac{n-3}{2} + \frac{1-n}{p}}
\norm[g][L^p(\R^{n-1})], \quad 1 \leq p \leq 2.
\end{equation*}

Returning to the action of $K_2^\eps$ on functions in $\R^n$,
these estimates imply that
\begin{align}
\bignorm[K_2^\eps * f][p'] &\les 
\Bignorm[{\int_{-\infty}^\infty (1+|x_n-y_n|)^{2\alpha + \frac{n-3}{2} + \frac{1-n}{p}}
\norm[f(\,\cdot\,,y_n)][L^p(\R^{n-1})]\,dy_n}][L^{p'}(\R)] \notag \\
&\les \norm[f][p]
\end{align}
provided $2\alpha + \frac{n-3}{2} + \frac{1-n}{p} \leq \frac{2}{p} - 2$, or more simply
$1 \leq p \leq \frac{2n+2}{4\alpha+n+1}$.  The last step is a restatement of the
Hardy-Littlewood-Sobolev inequality in one dimension.

Summing over the $n$ pieces of the conical decomposition concludes the proof.

\section{Application to embedded resonances of $-\Delta + V$}
\label{sec:embedded}
Statements like Theorem~\ref{thm:Helmholtz} are useful for constraining the
spectral measure of Schr\"odinger operators $-\Delta + V$ with a scalar
perturbation $V \in L^r(\R^n)$.  We present a tidy application here; one can
find a more extensive set of tools and results in~\cite{IoSc06}.

Suppose $n \geq 3$ and $V \in L^{\frac{n+1}{2}}(\R^n)$ is real-valued.
It is known that
the free resolvent
$R_0^+(\lambda) = \lim_{\eps\to 0^+} (-\Delta -(\lambda+i\eps))^{-1}$
maps $L^{\frac{2n+2}{n+3}}(\R^n)$ to $L^{\frac{2n+2}{n-1}}(\R^n)$ for
each $\lambda > 0$~\cite{KeRuSo87}.  One may present the resolvent of the 
perturbed operator $H = -\Delta + V$ using identities such as
\begin{equation*}
R_V^+(\lambda) := \lim_{\eps \to 0^+} (H - (\lambda + i\eps))^{-1}
= (I + R_0^+(\lambda)V)^{-1} R_0^+(\lambda).
\end{equation*}

The mapping bounds for $R_0^+(\lambda)$ extend naturally to the
pertubed resolvent $R_V^+(\lambda)$ provided there exists a suitable
operator inverse for $(I + R_0^+(\lambda)V)$.  Under the given condition
$V \in L^{\frac{n+1}{2}}(\R^n)$, this is a compact perturbation of the identity
on $L^{\frac{2n+2}{n-1}}(\R^n)$.  By the Fredholm alternative, it only fails to
be invertible if there exists a function $g \in L^{\frac{2n+2}{n-1}}$ such that
$g = -R_0^+(\lambda)Vg$.

Such a function also has the property $(R_0^+(\lambda)Vg, Vg) = -(g, Vg) \in \R$,
where $(\,\cdot\,,\,\cdot\,)$ is the sesquilinear pairing between 
$L^{\frac{2n+2}{n-1}}$ and its dual.  The imaginary part of the left-hand pairing
is equal to 
$c \lambda^{-1/2}\norm[(Vg)^\wedge][L^2(d\sigma_{\sqrt{\lambda}})]^2$,
hence the Fourier transform of $Vg$ vansihes on the sphere radius $\sqrt{\lambda}$.
Furthermore, g is a solution of the Helmholtz equation
$(-\Delta - \lambda)g = -Vg$.

The statement of Theorem~\ref{thm:Helmholtz} can be modified to accommodate
any operator $-\Delta -\lambda$, $\lambda > 0$, by conjugating with dilations of 
order $\sqrt{\lambda}$.  Write $V = V_1 + V_2$, where $V_1$ is bounded and 
compactly supported, and $\norm[V_2][\frac{n+1}{2}] < \delta$ for a quantity
$\delta>0$ to be chosen in a moment.  We have
\begin{equation}
\norm[g][2] \leq C_{n,\lambda}\norm[Vg][\frac{2n+2}{n+5}]
\leq C_{n,\lambda}\big(\norm[V_1][\frac{n+1}{3}]\norm[g][\frac{2n+2}{n-1}]
 + \delta \norm[g][2]\big).
\end{equation}
If $C_{n,\lambda}\delta < \frac12$, the last term can be moved to the left side
of the inequality so that $\norm[g][2] \les \norm[g][\frac{2n+2}{n-1}]$.

The conclusion is that resonances cannot be embedded into the continuous
spectrum of $H$; only true eigenfunctions in $L^2$ are possible.
However it is also known that embedded eigenvalues do not exist if the
potential is real and belongs to $L^{\frac{n+1}{2}}(\R^n)$~\cite{KoTa06},
so in fact the spectrum of $H$ is purely absolutely continuous.

\section{Perturbed Helmholtz equation}

Theorem~\ref{thm:Helmholtz} admits a relatively easy
extension to the equation $(-\Delta + V - 1)u = f$.
Factorize the perturbed Helmholtz operator  as
\begin{equation*}
-\Delta + V - 1 = (I + VR_0^+(1))(-\Delta - 1)
\end{equation*}
where $R_0^+(1) = \lim_{\eps \to 0^+}(-\Delta - (1+i\eps))^{-1}$.
In this case the choice of resolvent continuations is unimportant,
as both $R_0^+(1)$ and $R_0^-(1)$ act the same when applied to
functions in the range of $-\Delta - 1$.
Then there should exist $L^2$ solutions of $(-\Delta + V - 1)u = f$
whenever $f = g + VR_0^+(1)g$  and the unperturbed
equation $(-\Delta -1)u = g$ has solutions in $L^2(\R^n)$.

Let $X_0$ be the subspace of functions in $L^p(\R^n)$ whose
Fourier transform vanishes on the unit sphere, as defined in the statement
of Theorem~\ref{thm:Helmholtz2}.  For $p=\frac{2n+2}{n+5}$, 
Theorem~\ref{thm:Helmholtz} indicates that the latter problem admits solutions
precisely when $g \in L^p(\R^n)$ also belongs to $X_0 \subset L^p(\R^n)$.  
The substance of Theorem~\ref{thm:Helmholtz2} is that the correspondence
between $g$ and $f$ is an isomorphism of subspaces of $L^p(\R^n)$.
This statement is proved below.

\begin{proposition}
Assume the conditions of Theorem~\ref{thm:Helmholtz2}, namely that
$p = \frac{2n+2}{n+5}$ and $V \in L^{\frac{n+1}{2}}(\R^n)$.
Let $J: X_0 \to L^p(\R^n)$ be the inclusion map.  The linear operator
$J + VR_0^+(1): X_0 \to L^p(\R^n)$ is an isomorphism onto its range.
\end{proposition}

\begin{proof}
The fact that is a bounded operator is a direct consequence of
Theorem~\ref{thm:Helmholtz}, which effectively states that
$R_0^+(1)$ is a bounded map from $X_0$ to $L^2(\R^n)$.
It is injective by the result in~\cite{KoTa06}, as 
$R_0^+(1)g$ would be an $L^2$ eigenfunction of $-\Delta + V - 1$
for any $g$ in the nullspace of $J + VR_0^+(1)$.

In fact $VR_0^+(1)$ is a compact operator from $X_0$ into 
$L^p(\R^n)$.  For smooth compactly supported $V$ it acts
compactly on the larger domain $L^p(\R^n)$. Approximating 
$V \in L^{\frac{n+1}{2}}(\R^n)$ preserves compactness 
of $VR_0^+(1)$ over the restricted domain $X_0$.

The argument that $(J + VR_0^+(1))X_0 \subset L^p(\R^n)$ is closed is nearly
identical to the analogous statement in the Fredholm Alternative.
Let $f_n = (J + VR_0^+(1))g_n$ be a sequence converging to $f \in L^p$.
If $g_n$ has a bounded subsequence, then by compactness
$VR_0^+(1)g_n$ has a convergent subsqeuence and so does
$g_n = f_n - VR_0^+(1)g_n$.  The limit point $g \in X_0$ satisfies
$(J + VR_0^+(1))g = f$.

If $\lim_{n \to \infty} \norm[g_n][p] = +\infty$, consider the normalized
functions $\tilde{g}_n = g_n/\norm[g_n][p]$.  This sequence satisfies
$(J + VR_0^+(1))\tilde{g}_n \to 0$, and by compactness there is a
convergent subsequence of $VR_0^+(1)\tilde{g}_n$ with limit $-g$.
Then the same subsequence of $\tilde{g}_n$ converges to $g$, which
has unit norm and belongs to the nullspace of $J + VR_0^+(1)$.
That would violate the injectivity property of the map.

Having ruled out unbounded (subsequences of) $g_n$,
it follows that $f \in (J + VR_0^+(1))X_0$ as in the first case, 
making the range closed.  By the closed graph thoerem, $J + VR_0^+(1)$
is then an isomorphism onto its range.
\end{proof}

\section{Extensions via Interpolation} \label{sec:interpolation}
The subspace of $L^p$ consisting of functions whose Fourier transform
vanishes on the unit sphere is not particularly well suited to interpolation.
The Fourier-vanishing condition not preserved by
lattice operations or by the complex-analytic families used in the
Riesz-Thorin theorem.  As a futher impediment, it is not obvious that
one can approximate each element by a sequence of simple functions
(or compactly supported functions, or Schwartz functions)
whose Fourier transforms also vanish on the sphere.

We able to prove Theorem~\ref{thm:B-Rextension}
via complex interpolation of
operators and some careful avoidance of the above obstacles.
Suppose $f$ and $g$ are simple functions with compact support.
Let $\tilde{S}^{z}$ be the ``analytic"  Bochner-Riesz operators
defined by
\begin{equation*}
\tilde{S}^z = \frac{1}{\Gamma(z+1)}S^z
\end{equation*}
for real-valued $z > -1$, and by analytic continuation to $z \in \Compl$.

The key observation is that for ${\rm Re}\, z > -2$, and for functions
whose Fourier transform vanishes on the unit sphere,
$\Gamma(z + 1)\tilde{S}^{z}f = S^z f$ (The singularity at $z = -1$ 
is removable in this case).
Proposition~\ref{prop:main} establishes the same observation about
``two-sided" Bochner-Riesz operators over the larger range
${\rm Re}\,z > -3$.

It is true by construction that the function $\la \tilde{S}^z f, g\ra$
is holomorphic in $z$ for any pair of simple functions $f$ and $g$.
This remains true by uniform convergence in
the halfplane ${\rm Re}\, z > -\frac12 - \frac{2\beta n}{n+1}$ if we
take limits to a generic element $f \in L^{\frac{2n+2}{n+1+4\beta}}$.
Then
\begin{equation*}
G(z) := \Gamma(z+1)\la \tilde{S}^z f,g\ra
\end{equation*}
is meromorphic over the same domain, with residues at the negative integers
determined by $\la \tilde{S}^{-k}f,g\ra$.  Since $\tilde{S}^{-1}$ agrees (up to
a scalar multiple) with convolution against $\check{\sigma}_1$, if we further
assume that $\hat{f}$ vanishes on the unit sphere then in fact the singularity
of $G(z)$ at $ z=-1$ is removable.
The slightly modified function 
\begin{equation*}
\tilde{G}(z) := (z+2)\Gamma(z+1)\la \tilde{S}^zf,g\ra = (z+2)G(z)
\end{equation*}
is meromorphic with poles at the negative integers $k \leq -3$.

Assuming once again that $\hat{f}$ vanishes on the unit sphere,
Theorem~\ref{thm:B-R} provides a bound on the line $z = -\beta + i\mu$,
\begin{equation}
\bignorm[\Gamma(1-\beta+i\mu)\tilde{S}^{-\beta+i\mu}f][2]
\les \norm[f][\frac{2n+2}{n+1+4\beta}].
\end{equation}
The constant does not depend on $\mu$ because any one of the 
Fourier multipliers $(1-|\xi|^2)^{i\beta}$ is an isometry on $L^2$.
Therefore
\begin{equation*}
|\tilde{G}(-\beta + i\mu)| \les (1+|\mu|) 
\norm[f][\frac{2n+2}{n+1+4\beta}] \norm[g][2].
\end{equation*}

On the line $z = -2\beta + i\mu$,
we need the following estimates.

\begin{prop} \label{prop:complexB-R}
Let $\frac12 < \beta < \frac32$ and $\beta \leq \frac{n+1}{4}$.  The inequality
\begin{equation}
\bignorm[(2-2\beta+i\mu)\Gamma(1-2\beta + i\mu) \tilde{S}^{-2\beta+i\mu}f 
][\frac{2n+2}{n+1-4\beta}]
\les (1+|\mu|)\norm[f][\frac{2n+2}{n+1+4\beta}]
\end{equation}
holds uniformly for all $\mu \in \R$ and all $f \in L^{\frac{2n+2}{n+1+4\beta}}(\R^n)$.

%If $1 \leq \beta < \min(\frac32,\frac{n+1}{4})$,
%then there exists a constant $C_\beta$ such that
%\begin{equation}
%\Bignorm[ \Gamma(1-2\beta+i\mu)\tilde{S}^{-2\beta+i\mu}f
%- \frac{C_\beta}{2-2\beta+i\mu}\tilde{S}^{-2}f][\frac{2n+2}{2+1-4\beta}]
%\les \norm[f][\frac{2n+2}{n+1+4\beta}]
%\end{equation}
%uniformly for all $\mu \in \R$ and all $f \in L^{\frac{2n+2}{2+1+4\beta}}(\R^n)$.
\end{prop}

\begin{proof}[Sketch of Proof]
Proposition~\ref{prop:complexB-R} follows from the same argument as the endpoint
Stein-Tomas theorem,
using the fact that the convolution kernel of $\tilde{S}(z)$ has an asymptotic
description
\begin{equation*}
\tilde{S}^z (|x|) \sim \frac{C_n}{|x|^{\frac{n+1}{2} + z}} 
\cos\Big(|x| - \frac{(n-3)\pi}{4} + \frac{\pi}{2}z\Big)
\end{equation*} 
for large $|x|$.  Note that for complex $z$, oscillations of the cosine function
in this formula have amplitude approximately $e^{(\pi/2){\rm Im}\, z}$.

For $z = -2\beta + i\mu$, the prefactor $(z+2)\Gamma(z+1)$
is dominated by $(1+|z|)e^{-(\pi/2){\rm Im}\,z}$, using Stirling's
approximation when $\mu$ is large, and the absence of poles
for ${\rm Re}\,z > -3$ when $\mu$ is small.
Hence the product $(z+2)\Gamma(1+z)\tilde{S}^{z}$
enjoys mapping estimates that are uniform in $\mu$ along this line.
\end{proof}

Consequently $|\tilde{G}(-2\beta + i\mu)| \leq (1+|\mu|) 
\norm[f][\frac{2n+2}{n+1+4\beta}] \norm[g][\frac{2n+2}{n+1+4\beta}]$.
If one constructs
$g_z$ to be a holomorphic family of simple functions
(as in Riesz-Thorin interpolation) that belong isometrically to
$L^2(\R^n)$ along the line ${\rm Re}\,z = -\beta$, and to
$L^{\frac{2n+2}{n+1+4\beta}}(\R^n)$ along the line ${\rm Re}\,z = -2\beta$.
it follows from the Three-Lines Theorem that
\begin{equation*}
|(z+2)\Gamma(z+1)\la\tilde{S}^zf, g_z\ra| \les (1 + |{\rm Im}\, z|) 
\norm[f][\frac{2n+2}{n+1+4\beta}]
\norm[g_z][\frac{2n+2}{n+1-4({\rm Re}\,z +\beta)}]
\end{equation*}
 For a fixed real value $\beta \leq \alpha \leq 2\beta$, one can arrange
for $g_{-\alpha}$ to be any simple function.
It follows by duality and density of simple functions that
\begin{equation*}
\norm[\Gamma(z+1)\tilde{S}^{-\alpha}f][\frac{2n+2}{n+1-4(\alpha-\beta)}]
\les \frac{1}{|2-\alpha|} \norm[f][\frac{2n+2}{n+1+4\beta}].
\end{equation*}

For $\alpha < 2$ and with the assumption that $\hat{f}$ vanishes on the unit sphere,
the left hand function is exactly $S^{-\alpha}f$.  This is the norm bound
claimed in Theorem~\ref{thm:B-Rextension}.

As a final note, we observe that estimates can also be made for $\alpha < \beta$
by interpolating between Theorem~\ref{thm:Helmholtz} and other known bounds
for Bochner-Riesz operators.  To give a simple example, $S^{-\alpha}$
maps $L^p(\R^2)$ to itself for each $\alpha < -\frac12$, and if $\hat{f}$
vanishes on the unit circle for a function $f \in L^1(\R^2)$
it is also true that $S^{-\frac34}f \in L^2(\R^2)$.
Complex interpolation then suggests that $S^0f \in L^q(\R^2)$ for all
$q > \frac54$.  This is a modest improvement over the generic
$L^1 \mapsto L^{4/3+}$ bound for the ball multiplier.  We do not claim that
the exponent $q = \frac54$ is sharp, and suspect that
the range of exponents can be extended further toward 1 by other
methods.

\bibliographystyle{abbrv}
\bibliography{MasterList}

\end{document}